\documentclass[10pt,a4paper]{amsart}
\usepackage[cp1252]{inputenc}
\usepackage[T1]{fontenc}
\usepackage[a4paper]{geometry}
\usepackage{amsmath}
\usepackage{amsfonts}
\usepackage{amssymb}
\theoremstyle{plain}
\newtheorem{thm}{Theorem}
\newtheorem*{thmS}{Dimension free bounds for the Hardy-Littlewood maximal operator}
\newtheorem*{fsi}{Fefferman-Stein inequalities}
\newtheorem{pr}{Proposition}
\newtheorem{lem}{Lemma}
\theoremstyle{remark}

\newtheorem*{rem}{Remark}
\theoremstyle{definition}

\begin{document}

\title{Dimension free bounds for the vector-valued Hardy-Littlewood maximal operator}
\author{Luc Deleaval and Christoph Kriegler}
\date{September, 2016}
 \keywords{Hardy-Littlewood maximal operator; dimension free bounds; vector-valued estimates; UMD Banach lattice; Grushin operator}
\subjclass[2010]{42B25; 43A85; 46B42}
\maketitle

 \begin{abstract}
In this article, we prove Fefferman-Stein inequalities in $L^p(\mathbb R^d;\ell^q)$ with bounds independent of the dimension $d$, for all $1<p,q<+\infty$. This result generalizes in a vector-valued setting the famous one by Stein for the standard Hardy-Littlewood maximal operator. We then extend our result by replacing $\ell^q$ with an arbitrary UMD Banach lattice. Finally, we prove  similar dimensionless inequalities in the setting of the Grushin operators.
\end{abstract}

%
%

\section{Introduction and statement of the results}

At the beginning of the 1980s, Elias Stein proved in \cite{stsquare} (the complete detailed proof is in the paper of Stein-Str\"omberg \cite{2st}) that the standard Hardy-Littlewood maximal operator, that is associated with Euclidean balls, satisfies $L^p(\mathbb R^d)$ estimates with constant independent of the dimension $d$ for every $p>1$. More precisely, if we denote by
$\mathcal M$ the Hardy-Littlewood maximal operator, initially defined for $f \in L^1_{\mathrm{loc}}(\mathbb R^d)$ by
\[
\mathcal Mf(x)=\sup_{r>0}\frac{1}{|B(x,r)|}\int_{B(x,r)}|f(y)|dy, \quad x \in \mathbb R^d,
\]
with $B(x,r)$ the Euclidean ball centered at $x$ of radius $r>0$ and  $|X|$ the Lebesgue measure of a Borel subset $X$ of $\mathbb R^d$, then Stein's result reads as follows.
\begin{thmS}
Let $1<p\leqslant+\infty$. If $f \in L^p(\mathbb R^d)$, then we have
\[
\|\mathcal Mf\|_{L^p(\mathbb R^d)}\leqslant C(p)\|f\|_{L^p(\mathbb R^d)},
\]
where $C(p)$ is a constant independent of $d$.
\end{thmS}

This result, which improves in a spectacular fashion the behavior previously known, has opened the way to the following program: is it possible to bound uniformly in dimension the constant appearing in Hardy-Littlewood type estimates for maximal operators associated with symmetric convex bodies? This topic has been studied by various authors during the period 1986-1990 (see the papers of Bourgain \cite{B1,B2,B4}, Carbery \cite{Car} and M\"uller \cite{Mul}), and has been recently renewed by further advances, especially due to Bourgain \cite{B3}. For a thorough exposition of this subject, we refer the reader to the recent survey \cite{DGM}. In fact, Stein's result has opened the way, beyond the case of maximal functions, of proving fundamental estimates in harmonic analysis in $\mathbb R^d$ with formulations with bounds independent of the dimension.

It is therefore quite surprising that the question of a dimensionless behavior of the constant in the vector-valued extensions of the Hardy-Littlewood maximal theorem, the so-called Fefferman-Stein inequalities \cite{fs}, has not been tackled. This is the main purpose of our paper. Let us first recall these inequalities.
\begin{fsi}
Let $1<p, q <+\infty$ and let $(f_n)_{n\geqslant1}$ be a sequence of measurable functions defined on $\mathbb R^d$. If $\bigl(\sum_{n=1}^{+\infty}|f_n(\cdot)|^q\bigr)^\frac{1}{q} \in {L}^p(\mathbb R^d)$, then we have
\[
\biggl\|\Bigl(\sum_{n=1}^{+\infty}|\mathcal Mf_n(\cdot)|^q\Bigr)^\frac{1}{q}\biggr\|_{{L}^p(\mathbb R^d)}\leqslant C(d,p,q)\biggl\|\Bigl(\sum_{n=1}^{+\infty}|f_n(\cdot)|^q\Bigr)^\frac{1}{q}\biggr\|_{{L}^p(\mathbb R^d)},
\]
where $C(d,p,q)$ is a constant independent of  $(f_n)_{n\geqslant1}$.
\end{fsi}

The proof given by Fefferman and Stein for their inequalities, mainly based on the Calder\'on-Zygmund decomposition (for a weak-type result), the Marcinkiewicz interpolation theorem and a suitable weighted inequality, leads to a constant which growths exponentially with $d$. Another approach, based on Banach-space valued singular integrals \cite{GCRDF} (see also \cite{grafakos}), does not achieve this dimensionless goal either. In this paper, we succeed in proving the following dimensionless result.
\begin{thm} \label{indepdim}
Let $1<p, q <+\infty$ and let $(f_n)_{n\geqslant1}$ be a sequence of measurable functions defined on $\mathbb R^d$. If $\bigl(\sum_{n=1}^{+\infty}|f_n(\cdot)|^q\bigr)^\frac{1}{q} \in {L}^p(\mathbb R^d)$, then we have
\[
\biggl\|\Bigl(\sum_{n=1}^{+\infty}|\mathcal Mf_n(\cdot)|^q\Bigr)^\frac{1}{q}\biggr\|_{{L}^p(\mathbb R^d)}\leqslant C(p,q)\biggl\|\Bigl(\sum_{n=1}^{+\infty}|f_n(\cdot)|^q\Bigr)^\frac{1}{q}\biggr\|_{{L}^p(\mathbb R^d)},
\]
where $C(p,q)$ is a constant independent of $d$ and $(f_n)_{n\geqslant1}$.
\end{thm}

We believe that dimension free vector-valued estimates for general symmetric convex bodies $B$ should be true as well, but certainly not in full generality for both $p$ and $B$. Sharp vector-valued estimates on maximal operators associated with (radial) Fourier multipliers might be a key step, among others, to obtain such dimension free bounds.

Let $\mathcal{S}(\mathbb R^d)$ be the Schwartz class of smooth functions $\phi$ such that $(1+|x|^k)\phi^{(l)}(x)$ is bounded on $\mathbb R^d$ for all integers $k,l\geqslant0$. As in the proof of the dimensionless result by Stein for the Hardy-Littlewood maximal operator, the main tool in our proof will be the following spherical maximal operator $\mathcal{M}_S$, initially defined for $f \in \mathcal S(\mathbb R^d)$ by
\[
\mathcal{M}_Sf(x)=\sup_{r>0}\biggl|\int_{S^{d-1}}f(x-ry)d\sigma(y)\biggr|, \quad x \in \mathbb R^d,
\]
where  $d\sigma$ denotes the normalized Haar measure on $S^{d-1}$, and for which we will prove in particular the following vector-valued estimates.

\begin{thm} \label{sph}
Let $d\geqslant3$ and let $\frac{d}{d-1}<p,q<d.$
Let $(f_n)_{n\geqslant1}$ be a sequence of measurable functions defined on $\mathbb R^d$. If $\bigl(\sum_{n=1}^{+\infty}|f_n(\cdot)|^q\bigr)^\frac{1}{q} \in L^p(\mathbb R^d)$, then we have
\[
\biggl\|\Bigl(\sum_{n=1}^{+\infty}|\mathcal{M}_Sf_n(\cdot)|^q\Bigr)^\frac{1}{q}\biggr\|_{L^p(\mathbb R^d)}\leqslant C(d,p,q)\biggl\|\Bigl(\sum_{n=1}^{+\infty}|f_n(\cdot)|^q\Bigr)^\frac{1}{q}\biggr\|_{L^p(\mathbb R^d)},
\]
where $C(d,p,q)$ is a constant independent of $(f_n)_{n\geqslant1}$.
\end{thm}

We point out that vector-valued estimates for $\mathcal{M}_S$ have been recently proved by Manna in \cite{Manna}, for the range $\frac{2d}{d-1}<p,q<+\infty$, by use of a convenient weighted inequality for $\mathcal M_S$. We believe that the range $\frac{d}{d-1}<p,q<+\infty$ is optimal for $d\geqslant 3$, and, also in the case $d=2$, this might be true, as in the scalar case, see \cite{BouSph}.

In fact, we shall prove a more general result than Theorem 1. In order to state it, let us recall that a Banach space $Y$ is called UMD space if the Hilbert transform
\[ H : L^p(\mathbb R) \to L^p(\mathbb R),\: Hf(x) = PV - \int_{\mathbb R} \frac{1}{x-y} f(y) dy, \]
extends to a bounded operator on $L^p(\mathbb R;Y),$ for some $1 < p < +\infty$ (see Section 5 for more details).
Here and in what follows, we denote $L^p(\mathbb R^d; Y)$ the Bochner-Lebesgue space, i.e. the space of (equivalence classes of) measurable functions $f : \mathbb R^d \to Y$ such that
\[ \|f\|_{L^p(\mathbb R^d;Y)}^p = \int_{\mathbb R^d} \| f(x) \|_Y^p dx < +\infty .\]
With the above reminder in mind, we can now state our second main result, which is concerned with the UMD lattice valued Hardy-Littlewood maximal operator $\mathcal M$ (see Section 5 for the precise definition).

\begin{thm}\label{thm-independent-dimension-UMD-lattice}
Let $1 < p < + \infty$ and $Y = Y(\Omega,\mu)$ be a UMD Banach lattice, consisting without loss of generality of measurable functions over $(\Omega,\mu).$
We have with notations $x \in \mathbb R^d$ and $y \in \Omega,$
\[ \biggl\| \: \Bigl\|\mathcal M\bigl(f(\cdot,y)\bigr)(x)\Bigr\|_{Y} \: \biggr\|_{L^p(\mathbb R^d)} \leq C(p,Y) \|f\|_{ L^p(\mathbb R^d;Y)}, \]
where $C(p,Y)$ is a constant independent of $d$ and  $f \in  L^p(\mathbb R^d;Y).$
\end{thm}

The above theorem contains as a particular case Theorem 1 since $\ell ^q$ is a UMD Banach lattice for $1 < q < + \infty$, but we have made the decision, for the reader's convenience, to first prove Theorem 1 which is certainly an enlightening step for readers not familiar with UMD Banach lattices.

As a consequence of our two previous theorems, we shall prove, in the setting of Grushin operators,  vector-valued dimension free estimates for both the maximal operator $\mathcal{M}_{CC}$ associated with the Carnot-Carath\'eodory distance and the maximal operator $\mathcal{M}_K$ associated with the Kor\'anyi pseudo-distance (see the final section for more details).

\begin{thm}
\label{thm-Grushin}
Let $1 < p,q < + \infty$.
Then $\mathcal{M}_{CC}$ and $\mathcal{M}_K$ extend to bounded operators on $L^p(\mathbb R^{d+1};\ell^q)$ and there exists a constant $C = C(p,q)$ independent of $d$ such that
\begin{align*}
\biggl\|\Bigl(\sum_{n=1}^{+\infty}|\mathcal{M}_{CC}f_n(\cdot)|^q\Bigr)^\frac{1}{q}\biggr\|_{{L}^p(\mathbb R^{d+1})} &\leqslant C(p,q)\biggl\|\Bigl(\sum_{n=1}^{+\infty}|f_n(\cdot)|^q\Bigr)^\frac{1}{q}\biggr\|_{{L}^p(\mathbb R^{d+1})}\\
\intertext{and}
\biggl\|\Bigl(\sum_{n=1}^{+\infty}|\mathcal{M}_Kf_n(\cdot)|^q\Bigr)^\frac{1}{q}\biggr\|_{{L}^p(\mathbb R^{d+1})} &\leqslant C(p,q)\biggl\|\Bigl(\sum_{n=1}^{+\infty}|f_n(\cdot)|^q\Bigr)^\frac{1}{q}\biggr\|_{{L}^p(\mathbb R^{d+1})}.
\end{align*}
In the same manner, if $Y$ is a UMD Banach lattice, then $\mathcal{M}_{CC}$ and $\mathcal{M}_K$ extend to bounded operators on $L^p(\mathbb R^{d+1};Y)$ with norm $C = C(p,Y)$ independent of $d.$
\end{thm}

We end this introduction with an overview of the sections.
In Section \ref{sec-preliminary}, we prove two preliminary results on vector-valued maximal operators associated with multipliers and a Hilbertian square function estimate that we need in the sequel.
The method of proof of the main Theorems \ref{indepdim} and \ref{thm-independent-dimension-UMD-lattice} uses vector-valued estimates for the spherical maximal operator.
The latter is then studied in Section \ref{sec-spherical}, and both a sharp Hilbertian $L^2(\mathbb R^d;\ell^2)$ estimate and a weaker $L^p(\mathbb R^d;\ell ^q)$ estimate are built together by means of complex interpolation to yield the desired spherical maximal operator estimates, stated in Theorem \ref{sph}.
The next two Sections \ref{sec-proof-thm-1} and \ref{sec-proof-thm-2} are devoted to the proofs of Theorems \ref{indepdim} and \ref{thm-independent-dimension-UMD-lattice} respectively.
They use a technique of descent in the spirit of the Calder\'on-Zygmund method of rotations, and the spherical maximal operator estimates established beforehand.
Low dimensional estimates in the general case of UMD-lattice valued $L^p$ spaces are covered by the recent work of Xu \cite{Xu2015}.
Finally, in Section \ref{sec-Grushin}, we will prove Theorem \ref{thm-Grushin}.

\section{Preliminary results}
\label{sec-preliminary}
For $\omega \in \mathcal S(\mathbb R^d)$ a radial function,  we shall denote by ${M}_\omega$  the maximal operator associated with the multiplier $\omega$ and initially defined for $f \in \mathcal S(\mathbb R^d)$ by
\[M_\omega f(x)=\sup_{r>0}\bigl|\bigl(\hat{f}\mathopen(\cdot\mathclose)\omega(r\,\cdot)\bigr)^\vee(x)\bigr|=\sup_{r>0}\bigl|\bigl(f*(\omega^\vee)_r\bigr)(x)\bigr|, \quad x \in \mathbb R^d,
\]
where ${}^\vee$ is the inverse Fourier transform and where, for suitable $\psi$, $\psi_r$ is the dilation of $\psi$, that is to say
\[\psi_r(x)=\frac{1}{r^{d}}\psi\Bigl(\frac{x}{r}\Bigr), \quad x \in \mathbb R^d.\]

The following proposition provides us Fefferman-Stein inequalities for maximal operators associated with such a multiplier $\omega$.
\begin{pr} \label{map} Let $1<p,q<+\infty$ and let $(f_n)_{n\geqslant1}$ be a sequence of measurable functions defined on $\mathbb R^d$. If $\bigl(\sum_{n=1}^{+\infty}|f_n(\cdot)|^q\bigr)^\frac{1}{q} \in L^p(\mathbb R^d)$, then we have
\[
\biggl\|\Bigl(\sum_{n=1}^{+\infty}|M_\omega f_n(\cdot)|^q\Bigr)^\frac{1}{q}\biggr\|_{L^p(\mathbb R^d)}\leqslant C(d,p,q)\biggl\|\Bigl(\sum_{n=1}^{+\infty}|f_n(\cdot)|^q\Bigr)^\frac{1}{q}\biggr\|_{L^p(\mathbb R^d)},
\]
where $C(d,p,q)$ is a constant independent of $(f_n)_{n\geqslant1}$.
\end{pr}

\begin{proof} We claim that $\omega$ has an integrable radially decreasing majorant $\Omega$ since $\omega^\vee$ is a Schwartz radial function. Therefore, we have (see Corollary 2.1.12. page 84 in \cite{grafakos}) for every $x \in \mathbb R^d$
\[
M_\omega f(x)=\sup_{r>0}\bigl|\bigl(f*(\omega^\vee)_r\bigr)(x)\bigr|\leqslant \|\Omega\|_{L^1(\mathbb R^d)} \mathcal M f(x).
\]
Thus, all we have to do to conclude is to use the standard version of the vector-valued estimates for the Hardy-Littlewood maximal operator.
\end{proof}

We point out that the proof above applies to the following weak-type result: if $1<q<+\infty$ and if  $\bigl(\sum_{n=1}^{+\infty}|f_n(\cdot)|^q\bigr)^\frac{1}{q} \in L^1(\mathbb R^d)$, then for every $\lambda>0$ we have
\[
\biggl|\biggl\{x \in \mathbb R^d: \Bigl(\sum_{n=1}^{+\infty}|M_{\omega}f_n(x)|^q\Bigr)^\frac{1}{q}>\lambda\biggr\} \biggr|\leqslant \frac{C(d,q)}{\lambda}\biggl\|\Bigl(\sum_{n=1}^{+\infty}|f_n(\cdot)|^q\Bigr)^\frac{1}{q}\biggr\|_{L^1(\mathbb R^d)},
\]
where $C(d,q)$ is a constant independent of $(f_n)_{n\geqslant1}$ and $\lambda$.

We now introduce a square function that is closely related to the previous maximal multiplier operator. For a (radial) function $\omega \in \mathcal S(\mathbb R^d)$, we denote by $g_\omega$ the square function associated with the multiplier $\omega$ and initially defined for $f \in \mathcal S(\mathbb R^d)$ by
\[g_\omega(f)(x)=\biggl(\int_0^{+\infty}\bigl|\bigl(\hat{f}(x)\omega(tx)\bigr)^\vee\bigr|^2\frac{dt}{t}\biggr)^\frac{1}{2}, \quad x \in \mathbb R^d.\]

If the multiplier $\omega$ is supported in an annulus, then we can give the following precise upper bound, where a Hilbertian structure is required.

\begin{pr} \label{squaref}  Let $r$ be a positive real number. Suppose that $\omega$ is supported in the annulus $\{x \in \mathbb R^d: r\leqslant |x|\leqslant \rho r\}$ (with $\rho>1)$ and is bounded by C.  Let $(f_n)_{n\geqslant1}$ be a sequence of measurable functions defined on $\mathbb R^d$. If $\bigl(\sum_{n=1}^{+\infty}|f_n(\cdot)|^2\bigr)^\frac{1}{2} \in L^2(\mathbb R^d)$, then we have
\[ \biggl\|\Bigl(\sum_{n=1}^{+\infty}|g_\omega(f_n)(\cdot)|^2\Bigr)^\frac{1}{2}\biggr\|_{L^2(\mathbb R^d)}\leqslant C\sqrt{\ln(\rho)}\,\biggl\|\Bigl(\sum_{n=1}^{+\infty}|f_n(\cdot)|^2\Bigr)^\frac{1}{2}\biggr\|_{L^2(\mathbb R^d)},\]
where $C$ is the same constant in both the hypothesis and conclusion of the proposition.
\end{pr}

\begin{proof} The proof is nearly obvious. Indeed, by using successively Fubini's theorem, Plancherel's theorem and Fubini's  theorem again, we have
\[
\int_{\mathbb R^d}\sum_{n=1}^{+\infty}|g_\omega(f_n)(x)|^2dx=\sum_{n=1}^{+\infty}\int_{\mathbb R^d}\biggl(\int_0^{+\infty}\bigl|\hat{f_n}(x)\omega(tx)\bigr|^2\frac{dt}{t}\biggr)dx.\]
For all $x \in \mathbb R^d\setminus\{0\}$ and all $n\geqslant 1$, we can write
\[
\int_0^{+\infty}\bigl|\hat{f_n}(x)\omega(tx)\bigr|^2\frac{dt}{t}=|\hat{f_n}(x)|^2\int_{\frac{r}{|x|}}^{\frac{\rho r}{|x|}}|\omega(tx)|^2\frac{dt}{t}\leqslant C^2\ln(\rho)|\hat{f_n}(x)|^2,
\]
thus
\[\int_{\mathbb R^d}\sum_{n=1}^{+\infty}|g_\omega(f_n)(x)|^2dx\leqslant C^2\ln(\rho)\sum_{n=1}^{+\infty}\int_{\mathbb R^d}|\hat{f_n}(x)|^2dx.\]
To conclude, it is now enough to use Plancherel's theorem and Fubini's theorem.
\end{proof}

\section{Vector-valued inequalites for the spherical maximal operator}
\label{sec-spherical}

In this section, we prove the  vector-valued inequalities for the spherical maximal operator, stated in Theorem \ref{sph}.  These estimates will be a key tool in the proof of our dimensionless results. Let us begin with the following remark.

\begin{rem}
The condition $\frac{d}{d-1}<p$ can be easily seen to be necessary. Indeed, it suffices to consider the following sequence
\[
f_1(x)=\begin{cases} \mathrm{e}^\frac{-1}{1-|x|^2} \mathrm{\ \, if \ \,} |x|<1 \\ 0 \mathrm{\ \, if \ \,} |x|\geqslant1,
\end{cases} \quad f_2=f_3=\ldots=0.
\]
Of course, $\bigl(\sum_{n=1}^{+\infty}|f_n(\cdot)|^q\bigr)^\frac{1}{q} \in L^p(\mathbb R^d)$ while
\[
\Bigl(\sum_{n=1}^{+\infty}|\mathcal{M}_Sf_n(\cdot)|^q\Bigr)^\frac{1}{q} \notin L^p(\mathbb R^d)
\]
for $p\leqslant\frac{d}{d-1}$ since for $|x|$ large enough,
\[
\Bigl(\sum_{n=1}^{+\infty}|\mathcal{M}_Sf_n(x)|^q\Bigr)^\frac{1}{q}\geqslant \frac{C(d)}{|x|^{d-1}}.
\]
However, the condition $p<d$ is not optimal. Indeed, in a recent paper, Manna has proved by means of a convenient weighted inequality that Theorem \ref{sph} is true for the range $\frac{2d}{d-1}<p,q<+\infty$.
\end{rem}

In order to prove Theorem \ref{sph}, we do not follow Stein's ideas for the scalar case, but rather those of Rubio de Francia in \cite{rubio}. More precisely, we shall dominate $\mathcal{M}_S$ by a series of  maximal multiplier operators $\mathcal{M}_S\leqslant\sum_{l=0}^{+\infty}M_{m_l}$ (where $m_l$ is a radial multiplier), and we shall establish, for each $M_{m_l}$, a sharp $L^2(\mathbb R^d;\ell^2)$ estimate and a weaker $L^p(\mathbb R^d;\ell^q)$ estimate. Then, we shall proceed by complex interpolation, and the range of $p,q,$ in Theorem \ref{sph} is then relevant for series convergence. For the $L^2(\mathbb R^d;\ell^2)$ case, we mainly use both the decay at infinity and a support property for $m_l$, and a precise upper bound for the $L^2(\mathbb R^d)$-norm of an associated square function. For the  $L^p(\mathbb R^d;\ell^q)$ case, we mainly use the standard Fefferman-Stein inequalities and the Funk-Hecke formula. For the reader's convenience, we shall give the complete detailed proofs, which owe a lot to \cite{DGM,grafakos,rubio}.

To begin, we note that the spherical maximal operator could be expressed as follows
\[
\mathcal{M}_Sf(x)=\sup_{r>0}\bigl|\bigl(\hat{f}\mathopen(\cdot\mathclose)m(r\,\cdot)\bigr)^\vee(x)\bigr|=\sup_{r>0}\bigl|\bigl(f*(m^\vee)_r\bigr)(x)\bigr|,\]
where the multiplier $m$ is given by
\[m(x)=\widehat{d\sigma}(x)=\frac{2\pi}{|x|^{\frac{d-2}{2}}}J_{\frac{d-2}{2}}(2\pi |x|),
\]
with $J_\alpha$ the Bessel function of order $\alpha$. In order to decompose this multiplier into radial pieces with localized frequencies, we consider  a smooth radial function $\varphi_0$ on $\mathbb R^d$ satisfying
\[
\varphi_{0}(x)=\begin{cases} 1 \mathrm{\ \, if \ \,} |x|\leqslant 1 \\ 0 \mathrm{\ \, if \ \,} |x|\geqslant2.
\end{cases}
\]
Then, for every positive integer $l$, we define
\[
\varphi_l(x)=\varphi_0(2^{-l}x)-\varphi_0(2^{1-l}x),\]
and we therefore introduce the following dyadic radial pieces associated with the multiplier $m$
\[
\forall l\geqslant0, \quad m_l=\varphi_lm.
\]
Since it is  obvious that $\sum_{l=0}^{+\infty}\varphi_l=1$, we claim that
\[m=\sum_{l=0}^{+\infty}m_l.\]
Consequently, we have the following pointwise inequality
\begin{equation} \label{ptw}
\mathcal{M}_Sf(x)\leqslant \sum_{l=0}^{+\infty}M_{m_l}f(x), \quad x\in \mathbb R^d,
\end{equation}
where $M_{m_l}$ is defined at the beginning of Section 2 by specializing $\omega$ to $m_l$.

\subsection*{$L^2(\mathbb R^d;\ell^2)$ and $L^p(\mathbb R^d;\ell^q)$ estimates  for $M_{m_l}$}

As claimed before, we shall establish an $L^2(\mathbb R^d;\ell^2)$ estimate and an $L^p(\mathbb R^d;\ell^q)$ estimate  for $M_{m_l}$, wtih $l\geqslant1$. As we shall see in the proof of Theorem \ref{sph}, the case $M_{m_0}$ will be covered by Proposition \ref{map}.

Let us begin with the $L^2(\mathbb R^d;\ell^2)$-result for $M_{m_l}$.

\begin{pr} \label{propl2} Let  $l\geqslant 1$ and let $(f_n)_{n\geqslant1}$ be a sequence of measurable functions defined on $\mathbb R^d$. If $\bigl(\sum_{n=1}^{+\infty}|f_n(\cdot)|^2\bigr)^\frac{1}{2} \in L^2(\mathbb R^d)$, then we have
\[\biggl\|\Bigl(\sum_{n=1}^{+\infty}|M_{m_l}f_n(\cdot)|^2\Bigr)^\frac{1}{2}\biggr\|_{L^2(\mathbb R^d)}\leqslant \frac{C(d)}{2^{\frac{l(d-2)}{2}}}\biggl\|\Bigl(\sum_{n=1}^{+\infty}|f_n(\cdot)|^2\Bigr)^\frac{1}{2}\biggr\|_{L^2(\mathbb R^d)},\]
where $C(d)$ is a constant independent of $l$ and $(f_n)_{n\geqslant1}$.
\end{pr}

\begin{proof}
Let $n\geqslant1$ and $l \geqslant 1$. By applying the well-known differentiation theorem for multiples of approximate identities (see Corollary 2.1.19. page 88 in \cite{grafakos}), we get for almost all $x \in \mathbb R^d$
\[
\bigl(f_n*(m_l^\vee)_r\bigr)(x)\to m_l(0)f_n(x)=0
\]
as $r$ goes to $0$.  We can therefore write for almost $x \in \mathbb R^d$

\begin{align*}
\bigl(f_n*(m_l^\vee)_r\bigr)^2(x)&=\int_0^r\frac{d}{dt}\Bigl(\bigl(f_n*(m_l^\vee)_t\bigr)^2(x)\Bigr)dt\\ &=2\int_0^r\bigl(f_n*(m_l^\vee)_t\bigr)(x)\bigl(f_n*(\tilde{m}_l^\vee)_t\bigr)(x)\frac{dt}{t},
\end{align*}
where we have set
\[\tilde{m}_l(x)=\langle x,\nabla m_l(x)\rangle.
\]
We now enlarge the domain of the integral to obtain
\[\bigl|\bigl(f_n*(m_l^\vee)_r\bigr)(x)\bigr|^2\leqslant 2\int_0^{+\infty}\bigl|\bigl(f_n*(m_l^\vee)_t\bigr)(x)\bigr|\bigl|\bigl(f_n*(\tilde{m}_l^\vee)_t\bigr)(x)\bigr|\frac{dt}{t},\]
and this previous inequality can be reformulated as follows
\[\bigl|\bigl(\hat{f_n}\mathopen(\cdot\mathclose)m_l(r\,\cdot)\bigr)^\vee(x)\bigr|^2\leqslant 2\int_0^{+\infty}\bigl|\bigl(\hat{f_n}\mathopen(\cdot\mathclose)m_l(t\,\cdot)\bigr)^\vee(x)\bigr|\bigl|\bigl(\hat{f_n}\mathopen(\cdot\mathclose)\tilde{m}_l(t\,\cdot)\bigr)^\vee(x)\bigr|\frac{dt}{t}.\]
We first take the supremum over all $r>0$ and then use the Cauchy-Schwarz inequality in order to get
\[\bigl(M_{m_l}f_n(x)\bigr)^2\leqslant 2\,g_{m_l}(f_n)(x)g_{\tilde{m}_l}(f_n)(x).\]
By summing  over $n$ and by using again the Cauchy-Schwarz inequality, we are led to
\[
\sum_{n=1}^{+\infty}|M_{m_l}f_n(x)|^2\leqslant 2\Bigl(\sum_{n=1}^{+\infty}|g_{{m}_l}(f_n)(x)|^2\Bigr)^\frac{1}{2}\Bigl(\sum_{n=1}^{+\infty}|g_{\tilde{m}_l}(f_n)(x)|^2\Bigr)^\frac{1}{2}.
\]
We now integrate over $\mathbb R^d$ and we use the Cauchy-Schwarz inequality to deduce that
\begin{multline*}
\biggl\|\Bigl(\sum_{n=1}^{+\infty}|M_{m_l}f_n(\cdot)|^2\Bigr)^\frac{1}{2}\biggr\|^2_{L^2(\mathbb R^d)}\leqslant \\ 2\biggl\|\Bigl(\sum_{n=1}^{+\infty}|g_{{m}_l}(f_n)(\cdot)|^2\Bigr)^\frac{1}{2}\biggr\|_{L^2(\mathbb R^d)}\biggl\|\Bigl(\sum_{n=1}^{+\infty}|g_{\tilde{m}_l}(f_n)(\cdot)|^2\Bigr)^\frac{1}{2}\biggr\|_{L^2(\mathbb R^d)}.
\end{multline*}
Let us note the following immediate inclusions
\begin{align*}
\mathrm{supp }({m}_l)&\subset\{x \in \mathbb R^d: 2^{l-1}\leqslant |x|\leqslant 2^{l+1}\} \\ \mathrm{supp }(\tilde{m}_l)&\subset\{x \in \mathbb R^d: 2^{l-1}\leqslant |x|\leqslant 2^{l+1}\}.
\end{align*}
Therefore, thanks to Proposition \ref{squaref}, it is now enough to prove the following inequalities
\begin{equation} \label{bess2}
\|{m}_l\|_{L^\infty(\mathbb R^d)}\leqslant \frac{C_1(d)}{2^{{\frac{l(d-1)}{2}}}}, \quad \|\tilde{m}_l\|_{L^\infty(\mathbb R^d)}\leqslant \frac{{C_2(d)}}{2^{{\frac{l(d-3)}{2}}}},
\end{equation}
where both $C_1(d)$ and ${C_2(d)}$ are constants independent of $l$, since we shall deduce
\begin{multline*}
\biggl\|\Bigl(\sum_{n=1}^{+\infty}|M_{m_l}f_n(\cdot)|^2\Bigr)^\frac{1}{2}\biggr\|^2_{L^2(\mathbb R^d)}\leqslant \\
2\Biggl(\frac{C_1(d)\sqrt{\ln 4}}{2^{{\frac{l(d-1)}{2}}}}\biggl\|\Bigl(\sum_{n=1}^{+\infty}|f_n(\cdot)|^2\Bigr)^\frac{1}{2}\biggr\|_{L^2(\mathbb R^d)}\Biggr)\Biggl(\frac{C_2(d)\sqrt{\ln 4}}{2^{{\frac{l(d-3)}{2}}}}\biggl\|\Bigl(\sum_{n=1}^{+\infty}|f_n(\cdot)|^2\Bigr)^\frac{1}{2}\biggr\|_{L^2(\mathbb R^d)}\Biggr).
\end{multline*}
Thus, we now turn to the proof of \eqref{bess2} in order to complete the proof of the proposition. The following well-known estimate for the Bessel function (see for instance page 238 in \cite{Askey})
\[\sup_{x\geqslant0}\,x^{1/2}|J_\alpha(x)|<+\infty\]
together with the following equality
\[
\frac{d}{dt}J_\alpha(t)=\frac{1}{2}\Bigl(J_{\alpha-1}(t)-J_{\alpha+1}(t)\Bigr),
\]
allow us to write for all $x \in \mathbb R^d\setminus\{0\}$
\[|{m}_l(x)|\leqslant\frac{C_1(d)}{|x|^{\frac{d-1}{2}}}, \quad |\tilde{m}_l(x)|\leqslant\frac{{C_2(d)}}{|x|^{\frac{d-3}{2}}}.\]
We claim that \eqref{bess2} is proved since both ${m}_l$ and $\tilde{m}_l$ are localized near $|x|\simeq 2^l$.
\end{proof}

We now turn to a weaker $L^p(\mathbb R^d;\ell^q)$-result for $M_{m_l}$.

\begin{pr}\label{propl1} Let $l\geqslant 1$ and let $1<p,q<+\infty$. Let $(f_n)_{n\geqslant1}$ be a sequence of measurable functions defined on $\mathbb R^d$. If $\bigl(\sum_{n=1}^{+\infty}|f_n(\cdot)|^q\bigr)^\frac{1}{q} \in L^p(\mathbb R^d)$, then we have\[
 \biggl\|\Bigl(\sum_{n=1}^{+\infty}|M_{m_l} f_n(\cdot)|^q\Bigr)^\frac{1}{q}\biggr\|_{L^p(\mathbb R^d)}\leqslant C(d,p,q)\,2^l\biggl\|\Bigl(\sum_{n=1}^{+\infty}|f_n(\cdot)|^q\Bigr)^\frac{1}{q}\biggr\|_{L^p(\mathbb R^d)},
\]
where $C(d,p,q)$ is a constant independent of $l$ and $(f_n)_{n\geqslant1}$.
\end{pr}

\begin{proof} If we prove that for all $x \in \mathbb R^d$ and $l\geqslant1,$
\begin{equation} \label{schw}
\bigl|m_l^\vee(x)\bigr|\leqslant C(d)\,\frac{2^l}{\bigl(1+|x|\bigr)^{d+1}},
\end{equation}
then we claim, thanks to Corollary 2.1.12. page 84 in \cite{grafakos}, that
\[
\sup_{r>0}\bigl|\bigl(f_n*(m_l^\vee)_r\bigr)(x)\bigr|\leqslant \tilde{C}(d)2^l\mathcal Mf_n(x),
\]
and the standard Fefferman-Stein inequalities for the Hardy-Littlewood maximal operator $\mathcal M$ allows us to conclude. Therefore, we are left with the task of establishing \eqref{schw} to which we now turn. Let $x \in \mathbb R^d$. We can write by use of the Funk-Hecke formula
\[m_l^\vee(x)=C(d)\int_{-1}^1\overline{\varphi_l^\vee}\bigl(\sqrt{|x|^2+1-2|x|t}\bigr)\sqrt{1-t^2}^{d-3}dt,
\]
where we have used the notation, for a radial function $f$ and for all $\xi \in \mathbb R^d$, $f(\xi)=\overline{f}(|\xi|)$.
Since we have $\varphi_l^\vee=\Psi_{2^{-l}}$, with $\Psi \in \mathcal S(\mathbb R^d)$ by the very definition of $\varphi_l$, then
\[
\bigl|m_l^\vee(x)\bigr| \leqslant C(d)\int_{-1}^1\frac{2^{dl}}{\bigl(1+2^l\sqrt{|x|^2+1-2|x|t}\bigr)^{d+2}}\sqrt{1-t^2}^{d-3}dt.
\]
We set \begin{align*}
I_{-1}(x)&=[-1,1]\cap\Bigl\{t \in \mathbb R:  \sqrt{|x|^2+1-2|x|t}\leqslant2^{-l}\Bigr\}
\\
I_j(x)&=[-1,1]\cap\Bigl\{t \in \mathbb R:2^{j-l}<\sqrt{|x|^2+1-2|x|t}\leqslant2^{j+1-l} \Bigr\}, \quad \forall j\geqslant0,
\end{align*}
in order to write
\begin{equation*}
\bigl|m_l^\vee(x)\bigr|\leqslant C(d)(\Sigma_{1}+\Sigma_{2}),
\end{equation*}
with
\begin{align*}
\Sigma_1&= \sum_{j=-1}^{l}\int_{I_j(x)}\frac{2^{dl}}{\Bigl(1+2^l\sqrt{|x|^2+1-2|x|t}\Bigr)^{d+2}}\sqrt{1-t^2}^{d-3}dt
\\
\Sigma_2&= \sum_{j=l+1}^{+\infty}\int_{I_j(x)}\frac{2^{dl}}{\Bigl(1+2^l\sqrt{|x|^2+1-2|x|t}\Bigr)^{d+2}}\sqrt{1-t^2}^{d-3}dt.
\end{align*}
Therefore, the inequality \eqref{schw} is true if we show that
\[\Sigma_1\leqslant C_1(d)\,\frac{2^l}{\bigl(1+|x|\bigr)^{d+1}}, \quad \Sigma_2\leqslant C_2(d)\,\frac{2^l}{\bigl(1+|x|\bigr)^{d+1}},\]
where both $C_1(d)$ and ${C_2(d)}$ are constants independent of $l$. First, let us remark that for $t \in I_j(x)$, $j\geqslant-1$, we have
\begin{equation} \label{keyo}
|x|\leqslant 2^{j+1-l}+1,
\end{equation}
since for $x \in \mathbb R^d$ and $t \in I_j(x)$ fixed, we have
\begin{align*}
|x|\leqslant\sqrt{(|x|-t)^2}+|t| &\leqslant\sqrt{(|x|-t)^2+(1-t^2)}+|t|\\ &=\sqrt{|x|^2+1-2|x|t}+|t|\leqslant 2^{j+1-l}+1.
\end{align*}
We first prove the desired estimate for $\Sigma_1$. The following trivial observation $
-1\leqslant j\leqslant l \Longrightarrow 2^{j+1-l}+1\leqslant 3$ together with \eqref{keyo} lead us to
\[\Sigma_1\leqslant 2^{dl}\chi_{{}_{B(0,3)}}(x)\sum_{j=-1}^{l}\int_{I_{j}(x)}\frac{\sqrt{1-t^2}^{d-3}}{\Bigl(1+2^l\sqrt{|x|^2+1-2|x|t}\Bigr)^{d+2}}dt,\]
where we denote by $\chi_X$ the characteristic function of the set $X$.  Moreover, for all $t \in I_j(x)$, $0\leqslant j \leqslant l$,
\[
\frac{1}{\Bigl(1+2^l\sqrt{|x|^2+1-2|x|t}\Bigr)^{d+2}}\leqslant \frac{1}{2^{j(d+2)}},
\]
and, since this inequality remains obviously true for $j=-1$, we obtain
\[\Sigma_1\leqslant 2^{dl}\chi_{{}_{B(0,3)}}(x)\sum_{j=-1}^{l}\biggl(\frac{1}{2^{j(d+2)}}\int_{I_{j}(x)}\sqrt{1-t^2}^{d-3}dt\biggr).\]
Now we claim that, for all $-1\leqslant j\leqslant l$,
\[
\int_{I_{j}(x)}\sqrt{1-t^2}^{d-3}dt\leqslant 2^{(j+1-l)(d-1)+1}.
\]
Indeed, we have
\begin{align*}
\int_{I_{j}(x)}\sqrt{1-t^2}^{d-3}dt&\leqslant \int_{[-1,1]\cap\bigl\{t \in \mathbb R:\sqrt{1-t^2}\leqslant2^{j+1-l} \bigr\}}\sqrt{1-t^2}^{d-3}dt \\
&\leqslant 2^{(j+1-l)(d-3)}\int_{[-1,1]\cap\bigl\{t \in \mathbb R:\sqrt{1-|t|}\leqslant2^{j+1-l} \bigr\}}dt,
\end{align*}
and the following obvious observation
\[
t \in [-1,1]\cap\bigl\{t \in \mathbb R:\sqrt{1-|t|}\leqslant2^{j+1-l} \bigr\} \Longrightarrow  1-{2^{2(j+1-l)}} \leqslant |t|\leqslant 1
\]
then leads us to
\[
\int_{I_{j}(x)}\sqrt{1-t^2}^{d-3}dt\leqslant 2^{(j+1-l)(d-3)+1}\int^{1}_{1-{2^{2(j+1-l)}}}dt=2^{(j+1-l)(d-1)+1}.
\]
Consequently, we have
\[
\Sigma_1\leqslant 2\sum_{j=-1}^{l}\frac{2^{dl}2^{(j+1-l)(d-1)}}{2^{j(d+2)}}\chi_{{}_{B(0,3)}}(x)=2^{d+l}\sum_{j=-1}^{l}\frac{\chi_{{}_{B(0,3)}}(x)}{2^{3j}}\leqslant C(d)\,\frac{2^{l}}{\bigl(1+|x|\bigr)^{d+1}},\]
and it remains to prove the same estimate for $\Sigma_2$.

Thanks to \eqref{keyo} and to the very definition of $I_j(x)$, we claim that
\[
\Sigma_2\leqslant 2^{dl}\sum_{j=l+1}^{+\infty}\frac{\chi_{{}_{B(0,2^{j+1-l}+1)}}(x)}{2^{j(d+2)}}\int_{I_j(x)}\sqrt{1-t^2}^{d-3}dt.
\]
The obvious inclusion $B(0,2^{j+1-l}+1)\subset B(0,2^{j+2-l})$, for all $j\geqslant l+1$, together with the fact that
\[
\int_{I_j(x)}\sqrt{1-t^2}^{d-3}dt\leqslant \int_{-1}^1\sqrt{1-t^2}^{d-3}dt\leqslant C(d),
\]
allow us to write
\[
\Sigma_2\leqslant C(d)2^{dl}\sum_{j=l+1}^{+\infty}\frac{\chi_{{}_{B(0,2^{j+2-l})}}(x)}{2^{j(d+2)}}\leqslant \frac{C(d)2^{dl}}{\bigl(1+|x|\bigr)^{d
+1}}\sum_{
j=l+1}^{+\infty}\frac{\bigl(1+2^{j+2-l}\bigr)^{d+1}}{2^{j(d+2)}},
\]
from which we deduce that
\begin{align*}
\Sigma_2&\leqslant \frac{C(d)2^l}{\bigl(1+|x|\bigr)^{d+1}}\sum_{j=l+1}^{+\infty}\frac{2^{(j-l)\bigl((d+1)-(d+2)\bigr)}}{2^{l\bigl((d+2)+1-d\bigr)}}\\ &\leqslant \frac{C(d)2^l}{\bigl(1+|x|\bigr)^{d+1}}\sum_{j=l+1}^{+\infty}\frac{2^{l-j}}{2^{3l}}\\ &\leqslant C(d)\,\frac{2^{l}}{\bigl(1+|x|\bigr)^{d+1}}.
\end{align*}

\end{proof}

\begin{rem}
We point out that the pointwise inequality \eqref{schw}, which implies that
\[
\sup_{r>0}\bigl|\bigl(f_n*(m_l^\vee)_r\bigr)(x)\bigr|\leqslant \tilde{C}(d)2^l\mathcal Mf_n(x),
\]
gives us the following weak-type $L^1(\mathbb R^d;\ell^q)$-result ($1<q<+\infty$), with a constant $2^l$: if  $\bigl(\sum_{n=1}^{+\infty}|f_n(\cdot)|^q\bigr)^\frac{1}{q} \in L^1(\mathbb R^d)$, then for every $\lambda>0$ we have
\[
\biggl|\biggl\{x \in \mathbb R^d: \Bigl(\sum_{n=1}^{+\infty}|M_{m_l}f_n(x)|^q\Bigr)^\frac{1}{q}>\lambda\biggr\} \biggr|\leqslant C(d,q)\,\frac{2^l}{\lambda}\biggl\|\Bigl(\sum_{n=1}^{+\infty}|f_n(\cdot)|^q\Bigr)^\frac{1}{q}\biggr\|_{L^1(\mathbb R^d)},
\]
where $C(d,q)$ is a constant independent of $(f_n)_{n\geqslant1}$ and $\lambda$.
\end{rem}

We are now in a position to prove Theorem \ref{sph}.

\begin{proof}[Proof of Theorem \ref{sph}]
We begin the proof by noting that for all $1<p,q<+\infty$, $l \geqslant 1$ and $(f_n(\cdot))_{n\geqslant1} \in L^p(\mathbb R^d;\ell^q)$, we have
\[
\biggl\|\Bigl(\sum_{n=1}^{+\infty}|M_{m_l} f_n(\cdot)|^q\Bigr)^\frac{1}{q}\biggr\|_{L^p(\mathbb R^d)} = \Bigl\| A_{m_l}\bigl((f_n(\cdot))_{n\geqslant1}\bigr) \Bigr\|_{L^p(\mathbb R^d;\ell^q(L^\infty(\,]0,+ \infty[\,)))}, \]
where we have set
\[
A_{m_l} : \begin{cases} L^p(\mathbb{R}^d;\ell^q) \to L^p(\mathbb R^d;\ell^q(L^\infty(\,]0,+ \infty[\,))) \\ (f_n(\cdot))_{n\geqslant1} \mapsto \bigl(r \mapsto f_n \ast ( m_l ^\vee)_r\bigr)_{n\geqslant 1}.\end{cases}
\]
Proposition \ref{propl2} yields that
\[ \|A_{m_l}\|_{L^2(\mathbb{R}^d;\ell^2) \to L^2(\mathbb R^d;\ell^2(L^\infty(\,]0,+ \infty[\,)))} \leqslant C(d) 2^{-\frac{l(d-2)}{2}}, \]
whereas  Proposition \ref{propl1} yields
\[ \|A_{m_l}\|_{L^{p_0}(\mathbb{R}^d;\ell^{q_0}) \to L^{p_0}(\mathbb R^d;\ell^{q_0}(L^\infty(\,]0,+ \infty[\,)))} \leqslant C(d,p_0,q_0) 2^l\]
for any $1 < p_0, q_0 < + \infty$.
Complex interpolation between these two estimates yields for $\frac1p = \eta \frac12 + (1 - \eta) \frac1{p_0}$ and $\frac1q = \eta \frac12 + (1 - \eta) \frac1{q_0}$ that
\[ \|A_{m_l}\|_{L^p(\mathbb{R}^d;\ell^q) \to L^p(\mathbb R^d;\ell^q(L^\infty(\,]0,+ \infty[\,))) } \leqslant C(d,p,q) 2^{l ( 1 - \frac{d}{2} \eta )}. \]
Note that for given $\frac{d}{d-1} < p,q < d,$ we can choose $\eta \in ]0,1[$ such that
\[ \frac1d < \frac{\eta}{2} < \frac1p, \frac1q < 1 - \frac{\eta}{2} < \frac{d-1}{d} . \]
The second and third inequality yield that the parameters $p_0,q_0$ lie in the permitted range $]1, + \infty[,$ whereas the first and the fourth inequality yield that
$1 - \frac{d}{2} \eta < 0,$ so that
\[\sum_{l =1}^{+ \infty} \|A_{m_l}\|_{L^p(\mathbb{R}^d;\ell^q) \to L^p(\mathbb R^d;\ell^q(L^\infty(\,]0,+ \infty[\,))) } \leqslant C(d,p,q) 2^{l(1- \frac{d}{2} \eta)} < + \infty,\] and consequently, appealing also to \eqref{ptw} and Proposition \ref{map}, we get
\begin{multline*}
 \biggl\|\Bigl(\sum_{n=1}^{+\infty}|\mathcal{M}_Sf_n(\cdot)|^q\Bigr)^\frac{1}{q}\biggr\|_{L^p(\mathbb R^d)} \leqslant  \biggl( \|A_{m_0}\|_{L^p(\mathbb{R}^d;\ell^q) \to L^p(\mathbb R^d;\ell^q(L^\infty(\,]0,+ \infty[\,))) } + \\ \sum_{l =1}^{+ \infty} \|A_{m_l}\|_{L^p(\mathbb{R}^d;\ell^q) \to L^p(\mathbb R^d;\ell^q(L^\infty(\,]0,+ \infty[\,))) } \biggr)\biggl\|\Bigl(\sum_{n=1}^{+\infty}|f_n(\cdot)|^q\Bigr)^\frac{1}{q}\biggr\|_{L^p(\mathbb R^d)},
\end{multline*}
and finally
\[
 \biggl\|\Bigl(\sum_{n=1}^{+\infty}|\mathcal{M}_Sf_n(\cdot)|^q\Bigr)^\frac{1}{q}\biggr\|_{L^p(\mathbb R^d)}\leqslant  C(d,p,q) \biggl\|\Bigl(\sum_{n=1}^{+\infty}|f_n(\cdot)|^q\Bigr)^\frac{1}{q}\biggr\|_{L^p(\mathbb R^d)}.
\]
\end{proof}

\section{Proof of Theorem \ref{indepdim}}
\label{sec-proof-thm-1}

This section is devoted to the proof of our first dimensionless result, that is Theorem \ref{indepdim}. We shall use several auxiliary operators, and, even if we shall follow the same strategy as in the scalar case (see \cite{DGM,2st}), we give complete detailed proofs for the reader's convenience.

Let us first introduce the following weighted maximal operator, depending on a parameter $k \in \mathbb N$,

\[
\mathcal{M}_{d,k}f(x)=\sup_{r>0}\frac{\int_{|y|\leqslant r}|f(x-y)|\,|y|^kdy}{\int_{|y|\leqslant r}|y|^kdy}, \quad x \in \mathbb R^d.
\]
It is enough to take polar coordinates in the definition of $\mathcal{M}_{d,k}$ in order to obtain the following pointwise inequality
\begin{equation}
\label{equ-weighted-maximal-spherical-maximal}
\mathcal{M}_{d,k}f(x)\leqslant \mathcal{M}_S|f|(x), \quad x \in \mathbb R^d.
\end{equation}
Therefore, if we apply Theorem \ref{sph}, we get that for $d\geqslant 3$, $d/(d-1)<p,q<d$ and every sequence $(f_n)_{n\geqslant1}$  of measurable functions defined on $\mathbb R^d$ such that $\bigl(\sum_{n=1}^{+\infty}|f_n(\cdot)|^q\bigr)^\frac{1}{q} \in L^p(\mathbb R^d)$
\begin{equation} \label{pass}
\biggl\|\Bigl(\sum_{n=1}^{+\infty}|\mathcal{M}_{d,k}f_n(\cdot)|^q\Bigr)^\frac{1}{q}\biggr\|_{L^p(\mathbb R^d)}\leqslant C(d,p,q)\biggl\|\Bigl(\sum_{n=1}^{+\infty}|f_n(\cdot)|^q\Bigr)^\frac{1}{q}\biggr\|_{L^p(\mathbb R^d)},
\end{equation}
where $C(d,p,q)$ is a constant independent of  $k$ and $(f_n)_{n\geqslant1}$. Now, we shall obtain Theorem \ref{indepdim} by lifting inequality \eqref{pass} in lower dimension $d'$ into $\mathbb R^d$ (with $d'\leqslant d$ and $k=d-d'$) by integrating over the Grassmannian of $d'$-planes in $\mathbb R^d$. This method of descent is in the spirit of the Calder\'on-Zygmund method of rotations.

We therefore decompose $\mathbb R^d$ as follows $\mathbb R^d=\mathbb R^{d'}\times\mathbb R^{d-d'}$ and for $x \in \mathbb R^d$, we write $x=(x_{d'},x_{d-d'})$ with $x_{d'} \in \mathbb R^{d'}$ and $x_{d-d'} \in \mathbb R^{d-d'}$. Besides, for all $\theta \in \mathcal O(d) = \{ \theta' \in \mathbb R^{d \times d}:\: |\theta'(x)| = |x| \text{ for all }x \in \mathbb R^d\}$ the orthogonal group, we introduce the following auxiliary operator
\[
\mathcal{M}^\theta_{d'}f(x)=\sup_{r>0}\frac{\int_{|y_{d'}|\leqslant r}\bigl|f\bigl(x-\theta(y_{d'},0)\bigr)\bigr|\,|y_{d'}|^{d-d'}dy_{d'}}{\int_{|y_{d'}|\leqslant r}|y_{d'}|^{d-d'}dy_{d'}}, \quad x \in \mathbb R^d.
\]

We shall need the following lemma, which provides us Fefferman-Stein inequalities for $\mathcal{M}^\theta_{d'}$ with bound independent of $\theta$ and $d$.

\begin{lem} \label{l1}
Let $d'\geqslant 3$ and $d'/(d'-1)<p,q<d'$. Let $(f_n)_{n\geqslant1}$  be a sequence of measurable functions defined on $\mathbb R^d$. If $\bigl(\sum_{n=1}^{+\infty}|f_n(\cdot)|^q\bigr)^\frac{1}{q} \in L^p(\mathbb R^d)$, then we have
\[
\biggl\|\Bigl(\sum_{n=1}^{+\infty}|\mathcal{M}^\theta_{d'}f_n(\cdot)|^q\Bigr)^\frac{1}{q}\biggr\|_{L^p(\mathbb R^d)}\leqslant C(d',p,q)\biggl\|\Bigl(\sum_{n=1}^{+\infty}|f_n(\cdot)|^q\Bigr)^\frac{1}{q}\biggr\|_{L^p(\mathbb R^d)},
\]
where  $C(d',p,q)$ is a constant independent of $d$, $\theta$ and $(f_n)_{n\geqslant1}$.
\end{lem}

\begin{proof}
 Since we have
\[
\biggl\|\Bigl(\sum_{n=1}^{+\infty}|\mathcal{M}^\theta_{d'}f_n(\cdot)|^q\Bigr)^\frac{1}{q}\biggr\|^p_{L^p(\mathbb R^d)}=\biggl\|\Bigl(\sum_{n=1}^{+\infty}|\mathcal M^{\mathrm{Id}}_{d'}(f_n\circ\,\theta)\circ\,\theta^{-1}(\cdot)|^q\Bigr)^\frac{1}{q}\biggr\|^p_{L^p(\mathbb R^d)},
\]
then we get by invariance
\[
\biggl\|\Bigl(\sum_{n=1}^{+\infty}|\mathcal{M}^\theta_{d'}f_n(\cdot)|^q\Bigr)^\frac{1}{q}\biggr\|^p_{L^p(\mathbb R^d)}=\biggl\|\Bigl(\sum_{n=1}^{+\infty}|\mathcal M^{\mathrm{Id}}_{d'}(f_n\circ\,\theta)(\cdot)|^q\Bigr)^\frac{1}{q}\biggr\|^p_{L^p(\mathbb R^d)},
\]
which is equivalent to
\begin{multline*}
\biggl\|\Bigl(\sum_{n=1}^{+\infty}|\mathcal{M}^\theta_{d'}f_n(\cdot)|^q\Bigr)^\frac{1}{q}\biggr\|^p_{L^p(\mathbb R^d)}\\
=
\int_{\mathbb R^{d-d'}}\biggl[\int_{\mathbb R^{d'}}\biggl(\sum_{n=1}^{+\infty}\bigl|\mathcal M_{d',d-d'}(f_n\circ\,\theta\,\circ g_{x_{d-d'}})(x_d')\bigr|^q\biggr)^\frac{p}{q}dx_{d'}\biggr]dx_{d-d'}
\end{multline*}
where we have set
\[
 g_{x_{d-d'}}:\mathbb R^{d'}\to\mathbb R^d, \ \  x_{d'}\mapsto (x_{d'},x_{d-d'}).
\]
If we now apply inequality \eqref{pass}, we are led to
\begin{multline*}
\biggl\|\Bigl(\sum_{n=1}^{+\infty}|\mathcal{M}^\theta_{d'}f_n(\cdot)|^q\Bigr)^\frac{1}{q}\biggr\|^p_{L^p(\mathbb R^d)} \leqslant \\ C(d',p,q)\int_{\mathbb R^{d-d'}}\biggl[\int_{\mathbb R^{d'}}\biggl(\sum_{n=1}^{+\infty}\bigl|(f_n\circ\,\theta\,\circ g_{x_{d-d'}})(x_d')\bigr|^q\biggr)^\frac{p}{q}dx_{d'}\biggr]dx_{d-d'}
\end{multline*}
and the expected result easily follows.
\end{proof}

We shall also need the following lemma, which relates the Hardy-Littlewood maximal operator to the operator $\mathcal{M}^\theta_{d'}$.

\begin{lem} \label{l2}
If we denote by $d\mu$ the normalized Haar measure on  $\mathcal O(d)$, then we have the following pointwise inequality
\[
\mathcal Mf(x)\leqslant \int_{\mathcal O(d)} \mathcal{M}^\theta_{d'}f(x)d\mu(\theta), \quad x \in \mathbb R^d.
\]

\end{lem}

\begin{proof}
If we prove the following equality
\begin{equation} \label{intorth}
\frac{\int_{|y|\leqslant r}|f(y)|dy}{\int_{|y|\leqslant r}dy}=\frac{\int_{\mathcal O(d)}\int_{|y_{d'}|\leqslant r}\bigl |f\bigl(\theta(y_{d'},0)\bigr)\bigr|\,|y_{d'}|^{d-d'}dy_{d'}d\mu(\theta)}{\int_{|y_{d'}|\leqslant r}|y_{d'}|^{d-d'}dy_{d'}},
\end{equation}
then the expected inequality follows easily. Indeed, the previous equality allows us to write
\begin{align*}
\frac{1}{|B(0,r)|}\int_{B(0,r)}|f(x-y)|dy&=\frac{\int_{\mathcal O(d)}\int_{|y_{d'}|\leqslant r}\bigl |f\bigl((x-\theta(y_{d'},0)\bigr)\bigr|\,|y_{d'}|^{d-d'}dy_{d'}d\mu(\theta)}{\int_{|y_{d'}|\leqslant r}|y_{d'}|^{d-d'}dy_{d'}} \\ & \leqslant \int_{\mathcal O(d)}\mathcal{M}^\theta_{d'}f(x)d\mu(\theta)
\end{align*}
and it is then enough to take the supremum over all $r>0$. Therefore, we are left with the task of establishing \eqref{intorth}. Of course, by density arguments, we can restrict ourselves to finite linear combinations of functions whose expression has the following form $f(x)=f_0(|x|)f_1(x')$, with $x=|x|x'$ and $x' \in S^{d-1}$. If we take polar coordinates in the left-side of \eqref{intorth} we get
\[
\frac{d}{r^d}\int_0^rf_0(t)t^{d-1}dt\int_{S^{d-1}}f_1(y')d\sigma(y'),
\]
and if we take polar coordinates in the right-side of \eqref{intorth} we get
\[
\frac{d}{r^d}\int_0^rf_0(t)t^{d-1}dt\int_{\mathcal O(d)}\int_{S^{d'-1}}f_1\bigl(\theta(y'_{d'},0)\bigr)d\sigma(y'_{d'})d\mu(\theta).
\]
Thus, we only have to prove the following equality
\begin{equation} \label{pass2}
\int_{S^{d-1}}f_1(y')d\sigma(y')=\int_{\mathcal O(d)}\int_{S^{d'-1}}f_1\bigl(\theta(y'_{d'},0)\bigr)d\sigma(y'_{d'})d\mu(\theta).
\end{equation}
Let us note that we can write
\[
\int_{\mathcal O(d)}\int_{S^{d'-1}}f_1\bigl(\theta(y'_{d'},0)\bigr)d\sigma(y'_{d'})d\mu(\theta)=\int_{S^{d-1}}f_1(\eta)d\nu(\eta),
\]
where we have used the notation $d\nu$ to denote the pushforward of $d\sigma\otimes d\mu$ under the map $(y'_{d'},\tau)\mapsto\theta(y'_{d'},0)$.  Since $d\nu$ is invariant under the left-action of $\mathcal O(d)$, we claim that  $d\nu=d\sigma$, and the proof is finished.
\end{proof}

With the two previous lemmas in mind, we are now in a  position to prove Theorem \ref{indepdim}.

\begin{proof}[Proof of Theorem \ref{indepdim}]
Let $1<p,q<+\infty$.

If $d\leqslant3$ or $d\leqslant \max(\frac{p}{p-1},\frac{q}{q-1})$ or $d \leqslant \max(p,q)$, there is nothing to do, that is, we invoke the standard methods to get Fefferman-Stein inequalities.

 We can therefore assume that $d\geqslant3$ and that $\frac{d}{d-1}<p,q<d$. Thus, we write $d=d'+(d-d')$ with
 \[
d'=\bigl\lfloor \max\bigl(2,p,q,p/(p-1),q/(q-1)\bigr) \bigr\rfloor+1.
\]
The expected result then follows easily by using both  Lemma \ref{l1} and Lemma \ref{l2}.
\end{proof}

\section{Proof of Theorem \ref{thm-independent-dimension-UMD-lattice}}
\label{sec-proof-thm-2}

This section is devoted to the proof of the dimension free estimate of the UMD lattice valued Hardy-Littlewood maximal operator.
In order to make the meaning of the latter precise, we recall the necessary background on UMD lattices, i.e. Banach lattices which are UMD spaces.
First, for a general treatment of Banach lattices and their geometric properties, we refer the reader to  Chapter 1 in \cite{LT79}.
Second, a Banach space $Y$ is called UMD space if the Hilbert transform
\[ H : L^p(\mathbb R) \to L^p(\mathbb R),\: Hf(x) = PV - \int_{\mathbb R} \frac{1}{x-y} f(y) dy, \]
extends to a bounded operator on $L^p(\mathbb R;Y),$ for some (equivalently for all) $1 < p < +\infty$, see  Theorem 5.1 in \cite{HvNVW}.
A UMD space is super-reflexive \cite{Al79}, and hence (almost by definition) B-convex.
Let in the following $Y$ be a UMD space which is also a Banach lattice.
By $B$-convexity, $Y$ is order continuous and therefore it can be represented as a lattice consisting of (equivalence classes of) measurable functions on some measure space $(\Omega,\mu),$ see 1.a, 1.b in \cite{LT79}.
If $Y$ is a UMD lattice and $1 < p < +\infty,$ then $L^p(\mathbb R^d;Y)$ is again a UMD lattice.

\begin{rem}
In the proof below, we will use frequently and tacitly the following almost trivial observation:
if $0 \leqslant M(f) \leqslant N(f)$ for some element $f \in L^p(\mathbb R^d;Y)$ and images $M(f),N(f)$ of (typically non-linear) mappings $M,N,$ belonging to $L^p(\mathbb R^d;Y),$ then $\|M(f)\|_{L^p(\mathbb R^d;Y)} \leqslant \|N(f)\|_{L^p(\mathbb R^d;Y)}.$
This follows promptly from lattice axioms.
\end{rem}

Since the UMD lattice $Y$ is order continuous, by Proposition 1.a.8 in \cite{LT79}, it is also order complete.
Now define the UMD lattice valued Hardy-Littlewood maximal operator and spherical maximal operator by
\begin{align*}
\mathcal Mf(x,\cdot) & = \sup_{r > 0} \frac{1}{|B(x,r)|} \int_{B(x,r)} |f(y,\cdot)| dy, & f \in L^p(\mathbb{R}^d) \otimes Y \subseteq L^p(\mathbb{R}^d;Y) \\
\intertext{and}
\mathcal{M}_S f(x,\cdot) & = \sup_{r > 0} \biggl| \int_{S^{d-1}} f(x - ry,\cdot) d\sigma(y) \biggr|, & f \in \mathcal{S}(\mathbb R^d) \otimes Y \subseteq L^p(\mathbb{R}^d;Y) .
\end{align*}
Using the order completeness of $L^p(\mathbb R^d;Y)$ and the scalar valued boundedness of $\mathcal M$ and $\mathcal{M}_S,$ it is not difficult to show that the above suprema exist in $L^p(\mathbb R^d;Y)$ in the lattice sense.
Thus, $\mathcal Mf$ (resp. $\mathcal{M}_Sf$) are well-defined elements of $L^p(\mathbb R^d;Y)$ for the above $f$ and $1 < p < +\infty$ (resp. $\frac{d}{d-1} < p < +\infty$).
In the sequel, we can restrict in statements of boundedness of maximal operators first to $f$ belonging to $L^p(\mathbb R^d) \otimes Y$ or to $\mathcal{S}(\mathbb R^d) \otimes Y$ to have a priori maximal functions belonging to $L^p(\mathbb R^d;Y).$
Then use the remark above in this section and get an inequality for such $f$
\[
\|\mathcal Mf\|_{L^p(\mathbb R^d;Y)} \leqslant C \|f\|_{L^p(\mathbb R^d;Y)},
\]
resp. with $\mathcal M$ on the l.h.s. replaced by $\mathcal{M}_S$, or the several other maximal operators we use in the proof.
Finally for general $f \in L^p(\mathbb R^d;Y)$ and an approximating sequence of $f$, say $(f_n)_{n\geqslant1}$ belonging to $\mathcal{S}(\mathbb R^d) \otimes Y$,
$\mathcal Mf$ can be well defined by $\mathcal Mf = \lim_n \mathcal Mf_n,$ due to
\[
\|\mathcal Mf_n - \mathcal Mf_m\| \leqslant \|\mathcal M(f_n-f_m)\| \leqslant C \|f_n-f_m\|.
\]
We thus end up with the desired estimate $\|\mathcal Mf\| \leqslant C \|f\|$ for all $f \in L^p(\mathbb R^d;Y),$ with the same constant $C$ as before.

\begin{proof}[Proof of Theorem \ref{thm-independent-dimension-UMD-lattice}]
Let $1<p<+\infty$. Since $Y$ is a UMD Banach lattice, say over the measure space $(\Omega,\mu),$ we claim, invoking  Corollary page 216 in \cite{RdF1986}, that there exists another UMD Banach lattice $Z$ defined on $(\Omega,\mu),$ such that
\[
Y = [Z,H]_\eta, \quad \mathrm{with\ }  H = L^2(\Omega,\mu) \mathrm{\ a \ Hilbert \ space},
\]
where $[\cdot,\cdot]_\eta$ is the complex interpolation method, with $\eta \in ]0,1[.$
Having a closer look at the proof of Corollary page 216 in \cite{RdF1986} (alternatively, by replacing $Z$ with $[Z,H]_{\eta_0}$ and using reiteration of complex interpolation), we can assume $\eta$ sufficiently close to $0$ such that $\frac{\eta}{2} < \frac1p < 1 - \frac{\eta}{2}.$

Let $\mathcal{M}_S$ be the spherical maximal operator. Our first aim is to show its boundedness on $L^p(\mathbb R^d;Y)$ for $d$ sufficiently large.
Write
\[\mathcal{M}_Sf(x,y) \leqslant \sum_{l = 0}^{+ \infty} M_{m_l}f(x,y)\] as in \eqref{ptw}, as soon as $d \geq 3.$
Here and in the sequel, $x \in \mathbb R^d$ and $y \in \Omega.$
For $X \in \{Z,H,Y\},$ we denote the linear operator
\[A_{m_l} : \begin{cases} L^p(\mathbb R^d;X) \to L^p(\mathbb R^d;X(L^\infty(\,]0,+ \infty[\,)))\\ f \mapsto f \ast (m_l^\vee)_r(x,y).\end{cases}\]
We clearly have the following equalities
\begin{align*}
M_{m_l}f(x,y) &= \|A_{m_l} f(x,y,\cdot)\|_{L^\infty(\,]0,+ \infty[\,)} \\
\|M_{m_l}f\|_{L^p(\mathbb R^d;X)}& = \|A_{m_l}f\|_{L^p(\mathbb R^d;X(L^\infty(\,]0,+ \infty[\,)))}.
\end{align*}
Moreover, we have
\[
M_{m_0}f(x,y)  \leqslant C(d) \mathcal Mf(x,y),
\]
and it is a well-known fact that the centered Hardy-Littlewood maximal operator satisfies
\[
\mathcal Mf(x)  \leqslant C(d) \mathcal{M}_\Delta |f|(x),
\]
where we have set
\[
\mathcal{M}_\Delta f(x) = \sup_{t > 0} |e^{t\Delta} f(x)|
\]
for the maximal operator associated with the heat (diffusion) semigroup on $\mathbb R^d.$
Then according to Theorem 2 in \cite{Xu2015}, $\mathcal M$ is bounded on $L^p(\mathbb R^d;X)$ for any UMD Banach lattice $X$.
We shall apply this fact with $X = Y$ and $X = Z.$

Namely, first, we have by the above
\[
\|M_{m_0}f\|_{L^p(\mathbb R^d;Y)} \leqslant C(d,p,Y) \|f\|_{L^p(\mathbb R^d;Y)}.\]
Since $H$ is a Hilbert space, we have by Proposition \ref{propl2} the inequality
\begin{equation}
\label{equ-average-Hilbert}
 \| A_{m_l} f \|_{L^2(\mathbb R^d;H(L^\infty(\,]0,+ \infty[\,)))} \leqslant \frac{C(d)}{2^{\frac{l(d-2)}{2}}} \|f\|_{L^2(\mathbb R^d;H)}.
\end{equation}
Furthermore, according to \eqref{schw}, we have
\[
\sup_{r > 0} |A_{m_l}f(x,y,r)| \leqslant C(d) 2^l \mathcal Mf(x,y),
\]
and thus,
\begin{equation}
\label{equ-average-UMD-lattice}
\| A_{m_l} f \|_{L^{p_0}(\mathbb R^d;Z(L^\infty(\,]0,+ \infty[\,)))} \leqslant C(d,p_0,Z) 2^l \| f \|_{L^{p_0}(\mathbb R^d;Z)},
\end{equation}
for $p_0 \in ]1,+ \infty[$ to be chosen later.
We next want to apply complex interpolation of \eqref{equ-average-Hilbert} and \eqref{equ-average-UMD-lattice}.
To this end, note that
\[
[Z(\mathrm L^\infty),H(L^\infty)]_\eta \hookrightarrow [Z,H]_\eta(L^\infty) = Y(L^\infty)
\]
contractively, according for instance to (1.1) in \cite{Fan} and Calder\'on's interpolation identification between the Calder\'on-Lozanovskii space $Z^{1-\eta}H^\eta$ and $[Z,H]_\eta = Y,$ see for instance page 215 in \cite{RdF1986}.
Then we get
\[ \|A_{m_l}f\|_{L^{p}(\mathbb R^d;Y(L^\infty(\,]0,+ \infty[\,)))} \leqslant C(d,p,\eta,Z) 2^{l k(\eta)}\|f\|_{L^{p}(\mathbb R^d;Y)},\]
with
\[
\frac{1}{p} = \frac{\eta}{2} + \frac{1 - \eta}{p_0} = \frac{1}{p_0} + \eta \Bigl( \frac12 - \frac1{p_0} \Bigr)
\]
and
\[
k(\eta) = - \eta \frac{d-2}{2} + (1 - \eta) = 1 - \eta \frac{d}{2}.
\]
Here, since $\eta$ was sufficiently close to $0,$ there exists an appropriate choice of $p_0 \in ]1,+ \infty[.$
If now $d > d_0 := \max(3,\lfloor 2/\eta \rfloor + 1),$ then $k(\eta) < 0,$ so that then
\begin{multline*}
\|\mathcal{M}_S f\|_{L^{p}(\mathbb R^d;Y)} \leqslant \\ \|A_{m_0}f\|_{L^{p}(\mathbb R^d;Y(L^\infty(\,]0,+ \infty[\,)))} + \sum_{l = 1}^{+ \infty} C(d,p,\eta,Z) 2^{l k(\eta)} \|f\|_{L^{p}(\mathbb R^d;Y)},
\end{multline*}
which yields to
\[
 \|\mathcal{M}_S f\|_{L^{p}(\mathbb R^d;Y)}\leqslant C'(d,p,Y) \|f\|_{L^{p}(\mathbb R^d;Y)}.
\]
We deduce that for $d' \geqslant d_0$
\[ \| \mathcal{M}^\theta_{d'} f \|_{L^p(\mathbb R^d;Y)} \leqslant C(d',p,Y) \|f\|_{L^p(\mathbb R^d;Y)}, \]
which can be proved as Lemma \ref{l1} using \eqref{equ-weighted-maximal-spherical-maximal}, and the above vector-valued boundedness of the spherical maximal operator, i.e. on $L^p(\mathbb R^{d'};Y).$
Then by Lemma \ref{l2}, as in the case $Y = \ell^q,$ we deduce with $d' = d_0$ and $d \geqslant d'$ that the centered Hardy-Littlewood maximal operator is bounded on $L^p(\mathbb R^d;Y),$
with bound $C(d',p,Y)$ independent of $d \geqslant d'.$

For dimensions $d < d',$ we invoke the above explained estimate of $\mathcal Mf(x) \leqslant C(d) \mathcal{M}_\Delta |f|(x).$
\end{proof}

\section{Application to the Grushin maximal operator}
\label{sec-Grushin}

In this section, we show that Theorems \ref{indepdim} and \ref{thm-independent-dimension-UMD-lattice} can be transferred to the context of Grushin operators. Initially studied by Grushin (see for example \cite{Grushin}),  these operators have received considerable attention in recent times, especially with some results on their harmonic analysis, see for instance \cite{DzJ,Li,2m}. Moreover, dimensionless type results have been in particular investigated, mainly for Riesz transforms associated with them, see \cite{sath}.

Let us recall the setting. The Grushin operator is given by
\[ \Delta_G = \sum_{i = 1}^d \frac{\partial^2}{\partial x_i^2} + |x|^2 \frac{\partial^2}{\partial u^2} = \sum_{i = 1}^d(X_i^2 + U_i^2) \]
on the space $\mathbb R^{d+1} = \mathbb R^d_x \times \mathbb R_u$, with
\[ |x|^2 = \sum_{i = 1}^d x_i^2,\quad X_i = \frac{\partial}{\partial x_i}, \quad U_i = x_i \frac{\partial}{\partial u}, \]
where the smooth vector fields $\{X_i,U_i\}_{1 \leq i \leq d}$ satisfy the H\"ormander condition.
We point out that the operator $\Delta_G$ is related to the Heisenberg group $\mathbb H_d$, since it is actually the image of a sub-Laplacian associated with $\mathbb H_d$ under a representation acting on functions on $\mathbb R^{d+1}$.
Let $d_{CC}$ denote the Carnot-Carath\'eodory distance associated with $\{ X_1, \ldots, X_d,U_1,\ldots,U_d\}$ (see for example \cite{VSCC}).
Then $(\mathbb R^{d+1},d_{CC},dm)$ is a space of homogeneous type, where $dm$ stands for the Lebesgue measure, which is not, however, translation invariant.
We define a further pseudo-metric on $\mathbb R^{d+1}.$
Namely, for $g = (x,u)$ and $g' = (x',u')$ belonging to $\mathbb R^d_x \times \mathbb R_u,$ we let
\[ d_K(g,g') = \sqrt{\sqrt{(|x|^2+|x'|^2)^2+(2|u-u'|)^2} - 2 \langle x,x'\rangle},\]
where $\langle\cdot,\cdot\rangle$ denotes the standard Euclidean scalar product. Then $d_K$ is a pseudo-distance on $\mathbb R^{d+1}$ (which is, in fact, equivalent to $d_{CC}$ \cite{Li}) related to the fundamental solution of $\Delta_G$ (that is to say Green's function).
We denote balls with respect to these two (pseudo)-distances by
 \[
 B_{CC}(g,r) = \{ g' \in \mathbb R^{d+1} :\: d_{CC}(g,g') \leqslant r \}
 \]
 and
 \[
 B_K(g,r) = \{ g' \in \mathbb R^{d+1} :\: d_K(g,g') \leqslant r \}.
 \]
This gives rise to the following Hardy-Littlewood maximal operators $\mathcal{M}_{CC}$ and $\mathcal{M}_K$, respectively and naturally given for $f \in L^1_{\mathrm{loc}}(\mathbb R^{d+1})$ by
\begin{align*}
\mathcal{M}_{CC}f(g) & = \sup_{r > 0} \frac{1}{|B_{CC}(g,r)|} \int_{B_{CC}(g,r)} |f(g')| dg', & \quad g \in \mathbb R^{d+1}, \\
\mathcal{M}_Kf(g) & = \sup_{r > 0} \frac{1}{|B_{K}(g,r)|}\int_{B_{K}(g,r)} |f(g')| dg', & \quad g \in \mathbb R^{d+1}.
\end{align*}
If $Y = Y(\Omega)$ is a UMD Banach lattice, these operators extend, for  $g \in \mathbb R^{d+1},\: \omega \in \Omega,\: f \in L^p(\mathbb R^{d+1}) \otimes Y$, by the formula
\[ (\mathcal{M}_{CC} f)(g,\omega) = \mathcal{M}_{CC} (f(\cdot,\omega))(g),\]
and similarly for $\mathcal{M}_K.$
This a priori definition of $\mathcal{M}_{CC}f$ and $\mathcal{M}_Kf$ yields a well-defined element in $L^p(\mathbb R^{d+1};Y),$ similarly to the remarks at the beginning of Section \ref{sec-proof-thm-2} concerning the Hardy-Littlewood maximal operator.
We can also restrict ourselves to $f \in L^p(\mathbb R^d) \otimes Y$ in the proof of Theorem \ref{thm-Grushin}, to which we proceed now.

\begin{proof}[Proof of Theorem \ref{thm-Grushin}]
For $\mathcal{M}_K,$ all we have to know is that (see (7.2) in \cite{Li})
\[
\mathcal{M}_Kf(x,u) \leqslant C \mathcal M_{\mathbb R^d}\bigl( \mathcal M_{\mathbb R} f(\cdot,u) \bigr)(x),
\]
where $\mathcal M_{\mathbb R^d}$ and $\mathcal M_{\mathbb R}$ stand respectively for the standard Hardy-Littlewood maximal operator on $\mathbb R^d$ and $\mathbb R.$
We can therefore apply the dimension free Theorems \ref{indepdim} and \ref{thm-independent-dimension-UMD-lattice}.

Then for $\mathcal{M}_{CC},$ it suffices to have in mind that for all $g, g' \in \mathbb R^{d+1}$ (see Propositions 5.1 and 5.2 in \cite{Li})
\[ d_K(g,g') \leqslant d_{CC}(g,g') ,\quad |B_{CC}(g,1)| \geqslant C |B_K(g,1)|\]
with a constant $C > 0$ independent of the dimension $d \in \mathbb N.$
Indeed, using for both $B = B_{CC}$ and $B = B_K$
\[
|B((x,u),r)| = r^{d+2} |B(\delta_{r^{-1}}(x,u),1)|
\]
where $\delta_{r^{-1}}(x,u) = (r^{-1}x,r^{-2}u)$  \cite{Li}, we deduce
\begin{align*}
\mathcal{M}_{CC}f(g) & = \sup_{r > 0} \frac{1}{|B_{CC}(g,r)|} \int_{B_{CC}(g,r)} |f(g')| dg'\\
& = \sup_{r > 0} \frac{1}{r^{d+2} |B_{CC}(\delta_{r^{-1}}g,1)|} \int_{B_{CC}(g,r)} |f(g')| dg' \\
& \leqslant \sup_{r > 0} \frac{1}{Cr^{d+2} |B_K(\delta_{r^{-1}} g,1)|} \int_{B_K(g,r)} |f(g')| dg',
\end{align*}
that is to say,
\[
\mathcal{M}_{CC}f(g) \leqslant \frac{1}{C}\mathcal{M}_Kf(g).
\]
Therefore, the statement for $\mathcal{M}_{CC}$ follows from that for $\mathcal{M}_K.$
\end{proof}

Another interesting case would be the centered Hardy-Littlewood maximal operator on the Heisenberg group, as studied e.g. in \cite{LiHeisenberg,Zien}.
After personal communication with Hong-Quan Li, we do not know whether this maximal operator admits dimension free $\ell^q$ or UMD lattice valued estimates.

\vspace{0.3cm}

 \textbf{Acknowledgments}. The authors wish to thank the referee for her/his careful reading of the manuscript and helpful comments and remarks which improved the quality of the paper.


\footnotesize{
\noindent Luc Deleaval \\
\noindent
Laboratoire d'Analyse et de  Math\'ematiques Appliqu\'ees (UMR 8050)\\
Universit\'e Paris-Est Marne la Vall\'ee \\
5, Boulevard Descartes\\
Champs sur Marne\\
77454 Marne la Vall\'ee Cedex 2\\
luc.deleaval@u-pem.fr\hskip.3cm
}

\vspace{0.2cm}

\footnotesize{
\noindent Christoph Kriegler \\
\noindent
Laboratoire de Math\'ematiques (UMR 6620)\\
Universit\'e Blaise-Pascal (Clermont-Ferrand 2)\\
Campus Universitaire des C\'ezeaux\\
3, place Vasarely\\
TSA 60026\\
CS 60026\\
63 178 Aubi\`ere Cedex, France\\
christoph.kriegler@math.univ-bpclermont.fr\hskip.3cm
}

\end{document}